\documentclass[11pt,reqno]{amsart}
\usepackage[]{hyperref}
\usepackage{amsthm,amsfonts,amsmath,amssymb,enumerate,longtable}
\newtheorem{thm}{Theorem}[section]

\newtheorem{ques}[thm]{Question}
\theoremstyle{definition}
\newtheorem{defn}[thm]{Definition}
\newtheorem{rem}[thm]{Remark}
\newtheorem{exam}[thm]{Example}

\numberwithin{equation}{section}
\usepackage{graphicx}
\usepackage[
 a4paper,
 left=2.5cm,
 right=2.5cm,
 top=3cm,
 bottom=3 cm,
 ]{geometry}

\begin{document}
\title[Answers to Kirk-Shahzad's questions on strong $b$-metric spaces]{Answers to Kirk-Shahzad's questions on strong $b$-metric spaces}

\author{Tran Van An}

\address{Department of Mathematics, Vinh University, Vinh City,
Nghe An, Viet Nam}

\email{andhv@yahoo.com}

\author{Nguyen Van Dung}

\address[]{Faculty of Mathematics and Information Technology Teacher Education, Dong Thap University, Cao Lanh City, Dong Thap Province, Viet Nam}

\email{nvdung@dthu.edu.vn}

\subjclass[2000]{Primary 47H10, 54H25; Secondary 54D99, 54E99}%
\keywords{}

\begin{abstract} In this paper, two open questions on strong $b$-metric spaces posed by Kirk and Shahzad~\cite[Chapter~12]{KS2014-8} are investigated. A counter-example is constructed to give a negative answer to the first question, and a theorem on the completion of a strong $b$-metric space is proved to give a positive answer to the second question.
\end{abstract}

\maketitle

\section{Introduction and preliminaries}

In 1993, Czerwik~\cite{SC1993} introduced the notion of a $b$-metric which is a generalization of a metric with a view of generalizing the Banach contraction map theorem.

\begin{defn}[\cite{SC1993}] Let $X$ be a non-empty set and $d: X \times X \longrightarrow [0,+\infty) $ be a function such that for all $x,y,z \in X$,
\begin{enumerate} \item $d (x,y) = 0$ if and only if $x =y$.

\item $d(x,y) = d(y,x)$.

\item $d(x,z) \le 2[ d(x,y) + d(y,z)]$.
\end{enumerate}
Then $d$ is called a \emph{$b$-metric} on $X$ and $(X,d)$ is called a \emph{$b$-metric space}.
\end{defn}

After that, in~1998, Czerwik~\cite{SC1998} generalized this notion where the constant 2 was replaced by a constant $s \ge 1$, also with the name \emph{$b$-metric}. In~2010, Khamsi and Hussain~\cite{KH2010} reintroduced the notion of a $b$-metric under the name \emph{metric-type}.

\begin{defn}[\cite{KH2010}, Definition~6] \label{159-81} Let $X$ be a non-empty set, $K >0$ and $D: X \times X \longrightarrow [0,+\infty) $ be a function such that for all $x,y,z \in X$,
\begin{enumerate} \item $D (x,y) = 0$ if and only if $x =y$.

\item $D(x,y) = D(y,x)$.

\item \label{159-81-3} $D(x,z) \le K [ D(x,y) + D(y,z)]$.
\end{enumerate}
Then $D$ is called a \emph{metric-type} on $X$ and $(X,D,K)$ is called a \emph{metric-type space}.
\end{defn}

\begin{defn}[\cite{KH2010}, Definition~7] Let $(X,D,K)$ be a $b$-metric space.
\begin{enumerate} \item A sequence $\{ x_n\}$ is called \emph{convergent} to $x$, written as $\lim\limits_{n\rightarrow\infty} x_n = x$, if $\lim\limits_{n\rightarrow\infty} D(x_n,x) =0$.

\item A sequence $\{ x_n\}$ is called \emph{Cauchy} if $\lim\limits_{n,m\rightarrow\infty} D(x_n,x_m) = 0$.

\item $(X,D,K)$ is called \emph{complete} if every Cauchy sequence is a convergent sequence.
\end{enumerate}
\end{defn}

From Definition~\ref{159-81}.\eqref{159-81-3}, it is easy to see that $K \ge 1$.
Also in~2010, Khamsi~\cite{MAK2010} introduced another definition of a metric-type where the condition~\eqref{159-81-3} in Definition~\ref{159-81} was replaced by $$D(x,z) \le K [ D(x,y_1) + \ldots + D(y_n,z)]$$ for all $x,y_1, \ldots , y_n,z \in X$, see \cite[Definition~2.7]{MAK2010}.
In the sequel, the metric-type in the sense of Khamsi and Hussain~\cite{KH2010} will be called a $b$-metric to avoid the confusion about the metric-type in the sense of Khamsi~\cite{MAK2010}. Note that every metric-type is a $b$-metric.

The same relaxation of the
triangle inequality in Definition~\ref{159-81} was also discussed in~2003 by Fagin \emph{et al.}~\cite{FKS2003}, who called this new
distance measure nonlinear elastic matching. The authors of that paper remarked that this measure had been used, for example, in~\cite{FS1998} for trademark shapes and in~\cite{CKCKP1991} to measure ice floes. In~2009, Xia~\cite{qX2009} used this semimetric distance to study the optimal transport path between probability
measures.

In recent times, $b$-metric spaces were studied by many authors, especially fixed point theory on $b$-metric spaces~\cite{ADKR2015}, \cite{DLTH2013}, \cite{HD2015},  \cite[Chapter~12]{KS2014-8}, \cite{KDH2013}. Some authors were also studied topological properties of $b$-metric spaces. In~\cite{ATD2015}, An \emph{et al.} showed that every $b$-metric space with the topology induced by its convergence is a semi-metrizable space and thus many properties of $b$-metric spaces used in the literature are obvious. Then, the authors proved the Stone-type theorem on $b$-metric spaces and get a sufficient condition for a $b$-metric space to be metrizable. Notice that a $b$-metric space is always understood to be a topological space with respect to the topology induced by its convergence and a $b$-metric need not be continuous \cite[Examples~3.9 \& 3.10]{ATD2015}. This fact suggests a strengthening of the notion of $b$-metric spaces which remedies this defect.

\begin{defn}[\cite{KS2014-8}, Definition~12.7] Let $X$ be a non-empty set, $K \ge 1$ and $D: X \times X \longrightarrow [0,+\infty) $ be a function such that for all $x,y,z \in X$,
\begin{enumerate} \item $D (x,y) = 0$ if and only if $x =y$.

\item $D(x,y) = D(y,x)$.

\item \label{159-81-3} $D(x,z) \le D(x,y) + KD(y,z)$.
\end{enumerate}
Then $D$ is called a \emph{strong $b$-metric} on $X$ and $(X,D,K)$ is called a \emph{strong $b$-metric space}.
\end{defn}

\begin{rem}[\cite{KS2014-8}, page 122] \label{180-87}
\begin{enumerate} \item Every strong $b$-metric is continuous.

\item  \label{180-87-2} Every open ball $B(a,r) = \{x \in X: D(a,x) <r \} $ of a strong $b$-metric space $(X,D,K)$ is open.
\end{enumerate}
\end{rem}

In \cite[Chapter~12]{KS2014-8}, Kirk and Shahzad surveyed $b$-metric spaces, strong $b$-metric spaces, and related problems. An interesting work was attracted many authors is to transform results of metric spaces to the setting of $b$-metric spaces. It is only fair to point out that some results seem to require the full use of the triangle inequality of a metric space. In this connection, Kirk and Shahzad \cite[page 127]{KS2014-8} mentioned an interesting extension of Nadler's theorem due to Dontchev and Hager~\cite{DH1994}. Recall that for a metric space $(X,d)$ and $A, B \subset X$, $x \in X$,  $$ \mathrm{ dist}(x,A) = \inf \{ d(x,a): a \in A \}  $$ $$\delta (A,B) = \sup \{ \mathrm{ dist}(x,A): x \in B\}$$ and  these notation are understood similarly on $b$-metric spaces.

\begin{thm}[\cite{KS2014-8}, Theorem~12.7] \label{182-99} Let $(X,d)$ be a complete metric space, $T: X \longrightarrow X$ be a map from $X$ into a non-empty closed subset of $X$, and $x_0 \in X$ such that
\begin{enumerate} \item \label{182-99-1} $\mathrm{ dist} (x_0,Tx_0) < r(1 -k)$ for some $r>0$ and some $k \in [0,1)$.

\item \label{182-99-2} $\delta (Tx \cap B(x_0,r), Ty) \le k d(x,y)$ for all $ x,y \in B(x_0,r)$.
\end{enumerate}
Then $T$ has a fixed point in $B(x_0,r)$.
\end{thm}

Based on the definition of $\delta(A,B)$ and the proof of Theorem \cite[Theorem~12.7]{KS2014-8}, the assumption~\eqref{182-99-2} in the above theorem is implicitly understood as

\emph{$\delta (Tx \cap B(x_0,r), Ty) \le k d(x,y)$ for all $ x,y \in B(x_0,r)$ and $Tx \cap B(x_0,r) \ne \emptyset $}.

The authors of~\cite{KS2014-8} did not know whether Theorem~\ref{182-99} holds under the weaker strong $b$-metric assumption. Explicitly, we have the following question.

\begin{ques}[\cite{KS2014-8}, page 128]  \label{182-98} Let $(X,D,K)$ be a complete strong $b$-metric space, $T: X \longrightarrow X$ be a map from $X$ into a non-empty closed subset of $X$, and $x_0 \in X$ such that
\begin{enumerate} \item $\mathrm{ dist} (x_0,Tx_0) < r(1 -k)$ for some $r>0$ and some $k \in [0,1)$.

\item $\delta (Tx \cap B(x_0,r), Ty) \le k  D(x,y)$ for all $ x,y \in B(x_0,r)$ and $Tx \cap B(x_0,r) \ne \emptyset$.
\end{enumerate}
Does the map $T$ have a fixed point in $B(x_0,r)$?
\end{ques}

Recall that a map $f: X \longrightarrow Y$ from a $b$-metric space $(X,D,K)$ into a $b$-metric space $(Y,D',K')$ is called an \emph{isometry} if $D'(f(x),f(y)) = D(x,y)$ for all $x,y \in X$. Also, a $b$-metric space $(X ^*,D ^*, K ^*)$ is called a \emph{completion} of the $b$-metric space $(X,D,K)$ if $(X ^*,D ^*,K ^*)$ is compete and there exists an isometry $f: X \longrightarrow X ^*$ such that $\overline{ f(X)} = X ^*$. A classical result is that every metric space is dense in a complete metric space. So, it is interesting to ask whether this result holds or not in the setting of strong $b$-metric spaces.

\begin{ques}[\cite{KS2014-8}, page 128] \label{180-97} Is every strong $b$-metric space dense in a complete strong $b$-metric space?
\end{ques}

Kirk and Shahzad \cite[page 128]{KS2014-8} commented that if the answer of Question~\ref{180-97} is positive, then every contraction map $f:X \longrightarrow X$ on a strong $b$-metric space $X$ may be extended to a contraction map $f:X ^* \longrightarrow X ^*$ on a complete strong $b$-metric space $X ^*$
which has a unique fixed point.
Ostrowski's theorem \cite[Theorem~12.6]{KS2014-8} then would provide a method for approximating this
fixed point.

In this paper, two above questions on strong $b$-metric spaces are investigated. A counter-example is constructed to give a negative answer to Question~\ref{182-98}, and a theorem on the completion of a strong $b$-metric space is proved to give a positive answer to Question~\ref{180-97}.

\section{Main results}

First, the following example gives a negative answer to Question~\ref{182-98}.

\begin{exam} \label{182-97} Let $X = \{1,2,3 \}$, $D: X \times X \longrightarrow  [0, +\infty )$ be defined by
$$D(1,1) = D(2,2) = D(3,3) =0, D(1,2) = D(2,1) = 2, D(2,3) = D(3,2) = 1, D(1,3) = D(3,1) = 6$$ and a map $T: X \longrightarrow X$ be defined by $$T1 = 2, T2= 3, T3 = 1.$$ Then
\begin{enumerate} \item \label{182-97-1} $(X,D,K)$ is a complete strong $b$-metric space with $K = 4$.

\item \label{182-97-2} $T$ and $(X,D,K)$ satisfy all assumptions of Question~\ref{182-98} with $x_0 = 1$, $r = 6$, $k = \cfrac{ 1}{2}$.

\item \label{182-97-3} $T$ has no any fixed point.
\end{enumerate}
\end{exam}

\begin{proof} \eqref{182-97-1}. For all $x,y \in X$, it follows from definition of $D$ that
$D(x,y) = D(y,x) $ and

$D(x,y) = 0$ if and only if $x =y$.

We also have
$$ D(1,3) + K D(3,2) = 6 + 4. 1 = 10 \ge 2 = D(1,2)$$ $$ KD(1,3) +  D(3,2) =4. 6 + 1 = 25 \ge 2 = D(1,2)$$
$$ D(1,2) + K D(2,3) = 2 + 4. 1 = 6  =6 = D(1,3)$$ $$ KD(1,2) +  D(2,3) =4. 2 + 1 = 9 \ge 6 = D(1,3)$$
$$ D(2,1) + K D(1,3) = 2 + 4. 6 =  26 \ge 1= D(2,3)$$ $$ KD(2,1) +  D(1,3) =4. 2 + 6 = 14 \ge 1 = D(2,3).$$

By the above, $D$ is a strong $b$-metric on $X$. Since $X$ is finite and discrete, $X$ is complete.  So, $(X,D,K)$ is a complete strong $b$-metric space with $K = 4$.

\eqref{182-97-2}. Since $TX = X$, $TX$ is a non-empty closed subset of $X$.
We have $$\mathrm{ dist}(x_0,Tx_0) = \mathrm{ dist}(1, T1) = \mathrm{ dist}(1,\{2\}) = D(1,2) = 2 $$ and $$r(1-k) = 6 \big( 1 - \cfrac{ 1}{2}\big) = 3 .$$ This proves that $\mathrm{ dist} (x_0,Tx_0) < r(1 -k)$.

We also have $B(x_0,r) = B(1,6) = \{ 1,2\} $.

If $x = y =1$, then $Tx \cap B(x_0,r) = \{ 2\}$ and $$\delta (Tx \cap B(x_0,r), Ty) = \delta ( \{ 2\}, \{ 2\} )= D(2,2) = 0 \le k D(x,y).$$

If $x = y =2$, then $Tx \cap B(x_0,r) = \emptyset$.

If $x = 1, y = 2$, then $$\delta (Tx \cap B(x_0,r), Ty)  = \delta ( \{ 2\}, \{ 3\} ) = D(2,3) = 1 = \cfrac{ 1}{2} D(1,2) = k D(x,y)  . $$

If $x = 2,y =1$, then $Tx \cap B(x_0,r)= \emptyset.$

By the above, $\delta (Tx \cap B(x_0,r), Ty) \le k D(x,y)$ for all $ x,y \in B(x_0,r)$ and $Tx \cap B(x_0,r) \ne \emptyset$.

\eqref{182-97-3}. By definition of $T$, we see that $T$ has no any fixed point.
\end{proof}

Next, the following theorem is a positive answer to Question~\ref{180-97}.

\begin{thm} \label{159-80} Let $(X,D,K)$ be a strong $b$-metric space. Then
\begin{enumerate} \item $(X,D,K)$ has a completion $(X ^*,D ^*,K)$.

\item The completion of $(X,D,K)$ is unique in the sense that if $(X^*_1, D^*_1,K_1)$ and
$(X^{*}_2, D^*_2, K_2)$ are two completions of $(X, D,K)$, then there is a bijective isometry $ \varphi: X^*_1 \longrightarrow X^*_2$ which restricts to the identity on $X$.
\end{enumerate}
\end{thm}

\begin{proof} Put $$\mathcal{ C} = \big\{ \{ x_n\}: \{ x_n\} \text{ is a Cauchy sequence in } (X,D,K) \big\} .$$
Define a relation $\sim $ on $\mathcal{ C}$ as follows:

$\{ x_n\} \sim \{y_n\} $ if and only if $\lim\limits_{n \rightarrow \infty} D(x_n,y_n) = 0$,
for all $\{ x_n\} $, $\{y_n\}$ $ \in \mathcal{ C}$.

The relation $\sim$ obviously satisfies reflexivity and symmetry. If $\{ x_n\} \sim \{y_n\} $ and $\{y_n\} \sim \{z_n\}$, then $\lim\limits_{n \rightarrow \infty} D(x_n,y_n) = \lim\limits_{n \rightarrow \infty} D(y_n,z_n) = 0$. Since $$0 \le D(x_n,z_n) \le D(x_n,y_n) + KD(y_n,z_n) $$ for all $n$, $\lim\limits_{n \rightarrow \infty} D(x_n,z_n) = 0$. Thus $\{ x_n\} \sim \{z_n\} $. Therefore, the relation $\sim$ is an equivalent relation on $\mathcal{ C}$.

Denote $$X ^* = \big\{x^* = [ \{ x_n\} ]: \{ x_n\} \in \mathcal{ C} \big\} $$ where $x^* = [\{ x_n\} ]$ is an equivalence class of $\{ x_n\}$ under the relation $\sim$, and define a function $D ^* : X ^* \times X ^* \longrightarrow \mathbb{R}$ by
\begin{equation} \label{180-95} D ^* (x ^* , y^* ) = \lim\limits_{n \rightarrow \infty} D(x_n,y_n).\end{equation} We see that, for all $n,m$
\begin{eqnarray*} D(x_n,y_n) &\le& K D(x_n,x_m) + D(x_m,y_n)\\
&\le &
K D(x_n,x_m) + D(x_m,y_m) + K D(y_m,y_n).
\end{eqnarray*} It implies that
\begin{equation} \label{180-99} D(x_n,y_n) -D(x_m,y_m) \le K \big[D(x_n,x_m) + D(y_m,y_n) \big] .
\end{equation}
Also
\begin{eqnarray*} D(x_m,y_m) &\le& K D(x_m,x_n) + D(x_n,y_n) \\
&\le& K D(x_n,x_m) + D(x_n,y_n) + K D(y_n,y_m).
\end{eqnarray*} Therefore,
\begin{equation} \label{180-98} D(x_m,y_m) -D(x_n,y_n) \le K \big[D(x_n,x_m) + D(y_m,y_n) \big] .
\end{equation}
It follows from~\eqref{180-99} and~\eqref{180-98} that
\begin{equation} \label{180-88}  | D(x_m,y_m) -D(x_n,y_n)| \le K \big[D(x_n,x_m) + D(y_m,y_n) \big].
\end{equation}
Taking the limit as $n,m \rightarrow \infty$ in~\eqref{180-88}, we get $\lim\limits_{ n,m \rightarrow \infty }| D(x_m,y_m) -D(x_n,y_n)| = 0$, that is, $\big\{ D(x_n,y_n) \big\} $ is a Cauchy sequence in $\mathbb{R}$. Thus $\lim\limits_{n \rightarrow \infty} D(x_n,y_n)$ exists.

Moreover, if $\{ x_n\} \sim \{z_n \}$ and $\{y_n\} \sim \{ w_n\}$, then
\begin{equation} \label{180-94} \lim\limits_{n \rightarrow \infty} D(x_n,z_n) = \lim\limits_{n \rightarrow \infty} D(y_n,w_n) = 0.
\end{equation} We see that
\begin{eqnarray*} D(x_n,y_n) & \le & K D(x_n,z_n) + D(z_n,y_n)\\
& \le & K D(x_n,z_n) + D(z_n,w_n) + K D(w_n,y_n).
\end{eqnarray*}
It implies that $$D(x_n,y_n) - D(z_n,w_n) \le K D(x_n,z_n) + K D(w_n,y_n).$$ Similarly, $$D(z_n,w_n) - D(x_n,y_n) \le K D(z_n,x_n) + K D(y_n,w_n).$$ Therefore,
\begin{equation} \label{180-93} |D(x_n,y_n) - D(z_n,w_n)| \le K D(x_n,z_n) + K D(w_n,y_n).
\end{equation}
Taking the limit as $n,m \rightarrow \infty$ in~\eqref{180-93} and using~\eqref{180-94}, we get $\lim\limits_{n \rightarrow \infty} |D(x_n,y_n) - D(z_n,w_n)| =0$. Thus $\lim\limits_{n \rightarrow \infty}D(x_n,y_n) = \lim\limits_{n \rightarrow \infty} D(z_n,w_n)$. Therefore, the function $D ^*$ is well-defined.

In the next, we shall prove that $(X ^*,D ^*,K) $ is a strong $b$-metric space. For all $x^*, y^* , z ^* \in X ^*$, we~have

$D ^*(x^*, y^* ) = \lim\limits_{n \rightarrow \infty} D(x_n,y_n) \ge 0$ since $D(x_n,y_n) \ge 0$ for all $n$.

$D ^* (x^*,y^*) = 0$ if and only if $\lim\limits_{n \rightarrow \infty} D(x_n,y_n) = 0$, that is, $\{ x_n\} \sim \{y_n\} $. It is equivalent to $x ^* = y^*$.

$D ^*(x^*,y^*) = \lim\limits_{n \rightarrow \infty} D(x_n,y_n) = \lim\limits_{n \rightarrow \infty} D(y_n,x_n) = D ^*(y^*, x^*)$ since $D(x_n,y_n) = D(y_n,x_n) $ for all~$n$.

$D ^*(x^*,z ^* ) = \lim\limits_{n \rightarrow \infty} D(x_n,z_n) \le \lim\limits_{n \rightarrow \infty} \big[D(x_n,y_n) + K D(y_n,z_n) \big] = D ^*(x^*, y^*) + K D ^*(y^*, z ^* )$.

So, $(X ^*,D ^*,K) $ is a strong $b$-metric space.

For each $x \in X$, put $f(x) = [ \{x,x,x, \ldots \}] \in X ^*$. We see that $f$ is an isometry from $(X,D,K)$ into $(X ^*,D ^*,K)$ since $$ D ^*\big( f(x), f(y) \big) = \lim\limits_{n \rightarrow \infty} D(x,y) = D(x,y)$$ for all $x,y \in X$.

Next, we will prove that $f(X)$ is dense in $X ^*$. If $x^* = [ \{ x_n\} ] \in X ^*$, then $\lim\limits_{n,m \rightarrow \infty} D(x_n,x_m) = 0$. For each $i \in \mathbb{N}$, there exists $n_0 ^ i$ such that $D(x_n,x_m) \le \cfrac{ 1}{i} $ for all $n,m \ge n_0 ^ i$. It implies that $$0 \le D ^* \big( f(x_{n_0 ^ i}), x^* \big) = \lim\limits_{n \rightarrow \infty} D(x_{n_0 ^ i}, x_n) \le \cfrac{ 1}{i}.$$ So $\lim\limits_{i \rightarrow \infty} D ^* \big( f(x_{n_0 ^ i}), x^* \big) = 0$. This proves that $\lim\limits_{i \rightarrow \infty} f(x_{n_0 ^ i}) = x^*$, that is, $f(X)$ is dense in $X ^*$.

Next, we will prove that $(X ^*,D ^*,K)$ is complete. Let $\{ x_n ^*\}$ be a Cauchy sequence in $X ^*$, where $x_n ^* = [ \{ x_i ^ n\}_i ]$ for some $\{ x_i ^ n\}_i \in \mathcal{ C}$. Then
\begin{equation} \label{180-91} \lim\limits_{n,m \rightarrow \infty} D ^* (x_n ^*,x_m ^* ) =0.
\end{equation}
Note that the open ball $B\big(x_n ^*, \cfrac{ 1}{Kn}\big)$ is open by Remark~\ref{180-87}.\eqref{180-87-2}. From the fact that $f(X)$ is dense in $X ^*$, for each $n$ there exists $y_n \in X$ such that
\begin{equation} \label{180-92} D ^* (f(y_n),x_n ^* ) < \cfrac{ 1}{Kn}.
\end{equation}
By~\eqref{180-92}, for all $n,m$, we have
\begin{eqnarray} \label{180-90} \nonumber D(y_n,y_m) & =& D ^* \big( f(y_n), f(y_m) \big) \\ \nonumber
& \le & K D ^*(f(y_n),x_n ^* ) + D ^*(x_n ^*,  f( y_m ) )\\ \nonumber
& \le & K D ^*(f(y_n),x_n ^* ) + D ^*(x_n ^*, x_m ^* ) +K D ^* (x_m ^* ,f( y_m ) )\\
& < & \cfrac{ 1}{n} + D ^*(x_n ^*, x_m ^* ) + \cfrac{ 1}{m} .
\end{eqnarray}
Taking the limit as $n,m \rightarrow \infty$ in~\eqref{180-90} and using~\eqref{180-91}, we get \begin{equation} \label{180-86} \lim\limits_{n,m \rightarrow \infty} D(y_n,y_m) = 0.
 \end{equation} Thus $\{y_n\}$ is a Cauchy sequence in $(X,D,K)$. Put $y^* = [\{y_n\} ] \in X ^*$. From~\eqref{180-92}, we have
\begin{eqnarray} \label{180-89} D ^* (x_n ^* ,y^* )
& \le & K D ^*(x_n ^* , f(y_n)) + D ^* (f(y_n), y^* ) \nonumber\\ \nonumber
& < & K \cfrac{ 1}{Kn} + \lim\limits_{m \rightarrow \infty} D(y_n,y_m)\\
& = & \cfrac{ 1}{n} + \lim\limits_{m \rightarrow \infty} D(y_n,y_m).
\end{eqnarray}
Taking the limit as $n \rightarrow \infty $ in~\eqref{180-89} and using~\eqref{180-86}, we have $\lim\limits_{n \rightarrow \infty} D ^* (x_n ^* ,y^* ) = 0$, that is, $\lim\limits_{n \rightarrow \infty} x_n ^* = y ^* $ in $(X ^*, D ^*, K)$. Therefore, $(X ^*, D ^*, K)$ is complete.

Finally, we prove the uniqueness of the completion. Let $(X^*_1, D^*_1,K_1)$ and
$(X^{*}_2, D^*_2, K_2)$ be two completions of $(X, D,K)$. For each $x^*_1 \in X_1 ^*$, there exists $\{ x_n \} \subset X$ such that $\lim\limits_{n \rightarrow \infty}f_1( x_n ) = x ^* _1$ where $f_1: X \longrightarrow X_1 ^*$ is an isometry. Since $\{f_1(x_n) \}$ is convergent, $\{f_1(x_n) \}$ is a Cauchy sequence in $X_1 ^*$. Since $f_1$ is an isometry, $\{x_n\}$ is a Cauchy sequence in $X$. Note that there exists $f_2: X \longrightarrow X_2 ^*$ which is also an isometry. Then $\{f_2 (x_n) \}$ is a Cauchy sequence in $X_2 ^*$ and thus there exists $x_2 ^* \in X_2 ^*$ such that  $\lim\limits_{n \rightarrow \infty} f_2(x_n) = x_2 ^*$. Define $\varphi: X_1 ^* \longrightarrow X_2 ^* $ by $\varphi (x_1 ^* ) = x_2 ^*$.

If $y_2 ^* \in X_2 ^*$, then $y_2 ^* = \lim\limits_{n \rightarrow \infty} f_2(y_n ) $ for some $\{ y_n\} \subset X $. Since $\{f_2(y_n) \}$ is convergent, $\{f_2(y_n) \}$ is a Cauchy sequence in $X_2 ^*$. Since $f_2$ is an isometry, $\{y_n\}$ is a Cauchy sequence in $X$. Also, $f_1$ is an isometry, $\{f_1( y_n)\}$ is a Cauchy sequence in $X_1 ^*$. Then there exists $y_1 ^* = \lim\limits_{n \rightarrow \infty}f_1( y_n).$ Therefore, $y_2 ^* = \varphi (y_1 ^*)$. This proves that $\varphi$ is bijective. Moreover, for every $x ^*, y^* \in X_1 ^*$ with $x^* = \lim\limits_{n \rightarrow \infty} f_1(x_n)$ and $y^* = \lim\limits_{n \rightarrow \infty} f_1(y_n)$, by using the continuity of $D_1 ^* $ and $D_2 ^*$,  we have $$ D_1 ^*(x^*, y^*) = \lim\limits_{n \rightarrow \infty} D_1 ^*(f_1(x_n) ,f_1(y_n)) = \lim\limits_{n \rightarrow \infty} D(x_n ,y_n)$$ $$ = \lim\limits_{n \rightarrow \infty} D_2 ^*(f_2(x_n) ,f_2(y_n)) = D_2 ^*(\varphi (x^*), \varphi (y^*)).$$
It implies that $\varphi$ is a bijective isometry $ \varphi: X^*_1 \longrightarrow X^*_2$ which restricts to the identity on~$X$.
\end{proof}

Finally, the following example shows that techniques used in the proof of Theorem~\ref{180-99} may not be applied to $b$-metric spaces.

\begin{exam}\label{180-96} Let $X = \big\{0,1, \cfrac{ 1}{2}, \ldots, \cfrac{ 1}{n}, \ldots \big\} $ and $$ D(x,y) = \left\{\begin{array}{ll}0 &\hbox{ if } x =y\\ 1 &\hbox{ if } x \ne y \in \{0,1 \}\\
|x -y| & \hbox{ if } x \ne y \in \{ 0\} \cup \big\{\cfrac{ 1}{2n}: n =1,2,\ldots \big\} \\
 4& \hbox{ otherwise. }
 \end{array}\right. $$
Then $D$ is a $b$-metric on $X$ with $K = \cfrac{ 8}{3}$ \cite[Example~3.9]{ATD2015}. Put $x_n =1$, $y_n = \cfrac{ 1}{2n} $, $z_n = 1$ and $w_n = 0$ for all $n$. Then $\{ x_n\}$, $\{y_n\}$, $\{ z_n\}$, $\{ w_n\}$ are Cauchy sequences and $\{ x_n\} \sim \{z_n \}$ and $\{y_n\} \sim \{ w_n\}$. However, $$\lim\limits_{n \rightarrow \infty} D(x_n,y_n) = \lim\limits_{n \rightarrow \infty} D\big(1, \cfrac{ 1}{2n} \big) = 4 \ne 1 = D(1,0) = \lim\limits_{n \rightarrow \infty} D(z_n,w_n).$$
This shows that the formula~\eqref{180-95} is not well-defined for the above $b$-metric $D$.
\end{exam}

Though the above example shows that that techniques used in the proof of Theorem~\ref{180-99} may not be applied to $b$-metric spaces, we do not know whether Theorem~\ref{159-80} fully extends to $b$-metric spaces. So, we conclude with the following question.

\begin{ques} Does every $b$-metric space have a completion?
\end{ques}


\end{document}